\documentclass{amsart}
\usepackage{amssymb}

\newcommand{\e}{\varepsilon}
\newcommand{\U}{\mathcal U}
\newcommand{\V}{\mathcal V}
\newcommand{\C}{\mathcal C}
\newcommand{\A}{\mathcal A}
\newcommand{\W}{\mathcal W}
\newcommand{\IR}{\mathbb R}
\newcommand{\IN}{\mathbb N}
\newcommand{\IB}{\mathbb B}
\newcommand{\IQ}{\mathbb Q}
\newcommand{\IZ}{\mathbb Z}
\newcommand{\pr}{\mathrm{pr}}
\newcommand{\w}{\omega}

\newtheorem{theorem}{Theorem}
\newtheorem{corollary}{Corollary}
\newtheorem{problem}{Problem}
\newtheorem{example}{Example}
\newtheorem{claim}{Claim}

\title[The connected countable spaces of Bing and Ritter are topologically homogeneous]{The connected countable spaces of Bing and Ritter\\ are topologically homogeneous}
\author{Iryna Banakh, Taras Banakh, Olena Hryniv, Yaryna Stelmakh}
\address{I.Banakh: Ya. Pidstryhach Institute for Applied Problems of Mechanics and Mathematics of National Academy of Sciences of Ukraine, Naukova 3b, 79060, Lviv, Ukraine}
\email{ibanakh@yahoo.com}
\address{T.Banakh: Ivan Franko National University of Lviv (Ukraine) and Jan Kochanowski University in Kielce (Poland)}
\email{t.o.banakh@gmail.com}
\address{O.Hryniv and Ya. Stelmakh: Ivan Franko National University of Lviv, Universytetska 1, 79000, Ukraine}
\email{ohryniv@gmail.com, yarynziya@ukr.net}

\begin{document}
\begin{abstract}
Answering a problem posed by the second author on Mathoverflow \cite{MO}, we prove that the connected countable Hausdorff spaces constructed by Bing \cite{Bing} and Ritter \cite{Rit76} are  topologically homogeneous.
\end{abstract}
\maketitle

In \cite{Bing} Bing presented a simple construction of a Hausdorff space which is countable and connected. This example is included to the book ``Counterexamples in Topology'' \cite[Ex.75]{SS}. The first (difficult) example of a connected countable Hausdorff space was constructed by Urysohn in \cite{Urysohn}. For some other examples of such spaces, see \cite{BMT}, 
\cite{Brown}, \cite{Cvid}, \cite{Golomb59}, \cite{Golomb61}, \cite{Gustin}, \cite{Kannan}, \cite{Kirch}, \cite{Law}, \cite{Martin}, \cite{Miches}, \cite{Miller}, \cite{OR}, \cite{Rit76}, \cite{Rit77}, \cite{Rit78},
\cite[Ex. 61,126]{SS}, \cite{Tz88}, \cite{Tz96}, \cite{Tz98}.

The Bing space $\IB$ is the rational half-plane $\{(x,y)\in\IQ\times\IQ:y\ge 0\}$ endowed with the topology $\tau$ consisting of the sets $U\subset\IB$ such that for every point $(a,b)\in U$ there exists $\e>0$ such that
$$\{(x,0)\in\IB:|x-(a-b/\sqrt{3})|<\e\}\cup
\{(x,0)\in\IB:|x-(a+b/\sqrt{3})|<\e\}\subset U.$$
Observe that the points $(a-b/\sqrt{3})$ and $(a+b\sqrt{3})$ are the base vertices of the equilateral triangle with vertex at $(a,b)$ and base on the line $\IR\times\{0\}$.

It is easy to see that for any rational numbers $a>0$ and $b$ the affine map $f:\IB\to\IB$, $f:(x,y)\mapsto (ax+b,ay)$, is a homeomorphism of the Bing space $\IB$. This implies that the homeomorphism group of the Bing space $\IB$ has at most two orbits: $\IB_0:=\IQ\times\{0\}$ and $\IB\setminus\IB_0$. On the other hand, the Bing space endowed with the topology inherited from the Euclidean plane is a countable metrizable space without isolated points and hence is  homeomorphic to the space $\IQ$ of rational numbers according to the classical theorem of Sierpi\'nski \cite[6.2.A(d)]{En}.

These observations motivates the following problem posed by the second author on Mathoverflow \cite{MO}.

\begin{problem}
Is the Bing space $\IB$ topologically homogeneous?
\end{problem}

In this paper we shall answer this problem affirmatively. Moreover, we shall prove that any bijection $f:A\to B$ between $\theta$-discrete subsets $A,B$ of the Bing space extends to a homeomorphism of $\IB$.

A subset $D$ of a topological space $X$ is called 
\begin{itemize}
\item {\em discrete}  if each point $x\in X$ has a neighborhood $O_x\subset X$ that contains at most one point of the set $D$;
\item {\em $\theta$-discrete} in $X$ if each point $x\in X$ has a neighborhood $O_x\subset X$ whose closure $\bar O_x$ contains at most one point of the set $D$;
\item {\em $\theta$-closed} in $X$ if each point $x\in X\setminus D$ has a neighborhood $O_x\subset X$ whose closure $\bar O_x$ does not intersect the set $D$.
\end{itemize}
It is easy to see that each $\theta$-discrete subset of a topological space $X$ (satisfying the separation axiom $T_1$) is discrete (and closed in $X$). A subset $D$ of a regular topological space $X$ is $\theta$-discrete if and only if it is discrete and closed in $X$.

On the other hand, we have 

\begin{example}\label{ex1} The Bing space $\IB$ contains  a $\theta$-closed discrete subset, which is not $\theta$-discrete.
\end{example}

\begin{proof}   Take any sequence $\{a_n\}_{n\in\w}\subset\IB\setminus\IB_0$ that converges to zero $(0,0)$ in the Euclidean topology on $\IB\subset\IQ\times\IQ$.
It can be shown that the set $A=\{(0,0)\}\cup\{a_n\}_{n\in\w}$ is $\theta$-closed and discrete in $\IB$. On the other hand, $A$ is not $\theta$-discrete (as the closure of any neighborhood of $0$ contains all but finitely many points of the set $A$).
\end{proof}


The main result of this paper is the following theorem that will be proved in Section~\ref{s:proof}.

\begin{theorem}\label{main1} Any bijection $f:A\to B$ between any $\theta$-discrete subsets $A,B$ of the Bing space extends to a homeomorphism $\bar f:\IB\to\IB$ of $\IB$.
\end{theorem}

Now we present some corollaries of this theorem.

A topological space $X$ is called {\em $n$-homogeneous} for a positive integer number $n$ if any bijection $f:A\to B$ between $n$-element subsets of $X$ can be extended to a homeomorphism of $X$. Observe that a topological space is topologically homogeneous if and only if it is 1-homogeneous. It is easy to see that the real line is 2-homogeneous but not 3-homogeneous and the circle is 3-homogeneous but not 4-homogeneous. The space $\IQ$ of rational numbers is $n$-homogeneous for every $n\in\IN$. Since each finite subset of a Hausdorff space is $\theta$-discrete, Theorem~\ref{main1} implies the following corollary.

\begin{corollary} The Bing space $\IB$ is $n$-homogeneous for every $n\in\IN$.
\end{corollary}

In \cite{Rit76} Ritter observed that the topology $\hat\tau$ on $\IB$ generated by the base consisting of regular open subsets of $\IB$ turns $\IB$ into a connected locally connected countable Hausdorff space. The obtained topological space $(\IB,\hat\tau)$ will be denoted by $\hat \IB$ and called the {\em Ritter space}. We recall that a set $U$ in a topological space $X$ is {\em regular open} if $U$ coincides with the interior of its own closure $\bar U$ in $X$. Observe that each homeomorphism of the Bing space $\IB$ remains a homeomorphism of the Ritter space $\hat\IB$.

\begin{corollary} Any bijection $f:A\to B$ between $\theta$-discrete subsets $A,B$ of the Ritter space $\hat \IB$ extends to a homeomorphism $\bar f:\hat\IB\to\hat\IB$ of $\IB$.
\end{corollary}

\begin{proof} Observe that the $\theta$-discrete sets $A,B$ of the Ritter space remain $\theta$-discrete in the topology of the Bing space $\IB$. By Theorem~\ref{main1}, the bijection $f:A\to B$ extends to a homeomorphism $\bar f:\IB\to\IB$ of the Bing space and $\bar f$ remains a homeomorphism of the Ritter space $\hat\IB$.
\end{proof}

\begin{corollary} The Ritter space $\hat\IB$ is $n$-homogeneous for every $n\in\IN$.
\end{corollary}

It is known that any homeomorphism $h:A\to B$ between closed nowhere dense subsets $A,B$ of the space $\IQ$ of rational numbers extends to a homeomorphism $\bar h$ of $\IQ$. In particular, any homeomorphism  between closed discrete subspaces of $\IQ$ extends to a homeomorphism  of $\IQ$. This homogeneity property does not hold for the Bing space.

\begin{example}\label{ex2} There exists a bijection $f:A\to B$ between $\theta$-closed discrete subsets $A,B$ of $\IB$ that cannot be extended to a homeomorphism of $\IB$.
\end{example}

\begin{proof} By Example~\ref{ex1}, the Bing space contains a $\theta$-closed discrete subset $A$ which is not $\theta$-discrete. Let $f:A\to B$ be any bijection on the subset $B=\IZ\times\{0\}$ of $\IB$. It is easy to see that the set $B$ is $\theta$-discrete and hence is $\theta$-closed and discrete. Since the $\theta$-discreteness is a topological property, the bijection $f$ cannot be extended to a homeomorphism of $\IB$.
\end{proof}

\section{Proof of Theorem~\ref{main1}}\label{s:proof}

Theorem~\ref{main1} will be proved by a standard back-and-forth argument. But first we need to introduce some notation and prove some auxiliary results.

A family $\U$ of subsets of a set $X$ is called {\em disjoint} if $U\cap V=\emptyset$ for any distinct sets $U,V\in\U$. We say that a cover $\U$ of a set $X$ {\em refines} a cover $\V$ of $X$ if each set $U\in\U$ is contained in some set $V\in\V$. For a disjoint cover $\U$ of a set $X$ and a point $x\in X$ by $\U(x)$ we denote the unique set $U\in\U$ containing $x$. For a subset $Z\subset X$ we put $\U(Z)=\bigcup_{x\in Z}\U(x)$. 

We endow the set $$\IQ+\sqrt{3}\IQ:=\{x+\sqrt{3}y:x,y\in\IQ\}$$ with the topology generated by the Euclidean metric $d(x,y)=|x-y|$. A subset $C$ of $\IQ+\sqrt{3}\IQ$ is {\em order-convex} if  for any points $u<v$ in $C$ the set $\{t\in\IQ+\sqrt{3}\IQ:u\le t\le v\}$ is contained in $C$.

\begin{claim}\label{cl1} Any open cover $\mathcal U$ of $\IQ+\sqrt{3}\IQ$ can be refined by a disjoint open cover consisting of order-convex sets.
\end{claim}

\begin{proof} Write the countable set $X=\IQ+\sqrt{3}\IQ$ as $X=\{x_n\}_{n\in\w}$.
Let $\mathcal B$ be the family of all order intervals $(a,b)_X=\{x\in X:a<x<b\}$ where $a,b\in \IR\setminus X$ and $a<b$. Observe that each set $B\in\mathcal B$ is closed-and-open in $X$ and no finite subfamily of $\mathcal B$ covers $X$.

By induction, for every $k\in\w$ choose a number $n_k\in\w$ and sets $U_k\in\U$,  $B_k\in\mathcal B$ such that $$n_k=\min\{i\in \w:x_i\notin \bigcup_{j<k}B_j\},\;\; x_{n_k}\in U_k, \mbox{ \ and \ }x_{n_k}\subset B_k\subset U_k\setminus \bigcup_{j<k}B_j.$$

We claim that $X=\bigcup_{k\in\w}B_k$. Assuming that $X\ne \bigcup_{k\in\w}B_k$, we can consider the number $m=\min\{i\in\w:x_i\notin \bigcup_{k\in\w}B_k\}$ and conclude that $\{x_i\}_{i<m}\subset \bigcup_{k<p}B_k$ for some $p\in\w$. It follows that $n_p=\min\{i\in\w:x_i\notin \bigcup_{k<p}B_k\}=m$ and hence $x_m=x_{n_p}\in B_p\subset\bigcup_{k\in\w}B_k$, which contradicts the choice of $m$. This contradiction shows that $\{B_k\}_{k\in\w}\subset\mathcal B$ is a disjoint cover of $X$, refining the cover $\U$.
\end{proof}

Consider two maps 
$$\pr_-:\IB\to\IQ+\sqrt{3}\IQ,\;\pr_-:(x,y)\mapsto x-\sqrt{3}y,\quad\mbox{and}\quad
\pr_+:\IB\to\IQ+\sqrt{3}\IQ,\;\;\pr_+:(x,y)\mapsto x+\sqrt{3}y.$$ 
For every $z\in \IB$ let $z_-=\pr_-(z)$, $z_+=\pr_+(z)$ and  $z_\pm:=\{z_-,z_+\}$.
For a subset $Z\subset\IB$ let $Z_\pm:=\bigcup_{z\in Z}z_\pm$.

\begin{claim}\label{cl2} For any $\theta$-discrete subset $D\subset\IB$ the set $D_\pm$ is closed and discrete in $\IQ+\sqrt{3}\IQ$.
\end{claim}

\begin{proof} Given any point $x\in\IQ+\sqrt{3}\IQ$, we should find a neighborhood $O_x\subset \IQ+\sqrt{3}\IQ$ of $x$ such that $|O_x\cap D_\pm|\le 1$. It follows that $x\in z_\pm$ for some point $z\in\IB$. Since the set $D$ is $\theta$-discrete, the point $z$ has a neighborhood $O_z\subset \IB$ whose closure has at most one common point with the set $D$. By the definition of the topology of the Bing space $\IB$, there exists $\e>0$ such that $U_x:=\{y\in \IB_0:|x-y|\le\e\}\subset O_z$. Observe that the closure $\bar U_x$ of the set $U_x$ in $\IB$ contains all points $b\in \IB$ such that $\min\{|x-b_-|,|x-b_+|\}\le\e$. Then $O_x:=\{y\in\IQ+\sqrt{3}\IQ:|x-y|<\e\}$ is a neighborhood of $x$ in $\IQ+\sqrt{3}\IQ$ such that
$$|O_x\cap D_\pm|\le |\bar U_x\cap D|\le |\bar O_z\cap D|\le 1.$$
\end{proof}


For two sets $U,V\subset\IQ+\sqrt{3}\IQ$ let $$U\diamond V:=\{z\in\IB:z_\pm\subset U\cup V,\;z_\pm\cap U\ne\emptyset\ne V\cap z_\pm\}.$$

Any disjoint cover $\U$ of $\IQ+\sqrt{3}\IQ$ induces the disjoint cover $$\U^\diamond=\{U\diamond V:U,V\in\U\}$$
of the Bing space $\IB$. If the cover $\U$ consists of order-convex sets, then $\U^\diamond$ consists of subsets of $\IB$ resembling ``rombes'' and ``triangles'' with baseline on $\IB_0$.

Any (bijective) function $\varphi:\U\to\V$ between disjoint covers $\U,\V$ of $\IQ+\sqrt{3}\IQ$ induces a (bijective) function $\varphi^\diamond:\U^\diamond\to\V^\diamond$ assigning to each set $U\diamond V\in\U^\diamond$ the set $\varphi(U)\diamond\varphi(V)$.
\smallskip

Now we are ready to present {\em the proof of Theorem~\ref{main1}}. Fix any bijection $f_0:A_0\to B_0$ between two $\theta$-discrete subsets $A_0$ and $B_0$ of the Bing space $\IB$.

Let $$\check A:=\{a\in A_0\setminus \IB_0:f_0(a)\in\IB_0\},\;\;\check B=\{b\in B_0\setminus \IB_0:f_0^{-1}(b)\in\IB_0\}$$and
$$\hat A:=\{a\in A_0\cap \IB_0:f_0(a)\notin\IB_0\},\;\;\hat B=\{b\in B_0\cap \IB_0:f^{-1}(b)\notin\IB_0\}.$$
It is clear $f_0(\check A)=\hat B$ and $f_0(\hat A)=\check B$. 

Let $\e$ be a positive real number and $A',B'$ be $\theta$-discrete sets in $\IB_0$ such that $A_0\subset A'$ and $B_0\subset B'$.

A cover $\U$ of $\IQ+\sqrt{3}\IQ$ is called {\em $(A',B';\e)$-admissible} if
\begin{enumerate}
\item $\U$ is disjoint and consists of closed-and-open subsets of $\IQ+\sqrt{3}\IQ$;
\item for each $z\in \check A\cup\check B$ there are two disjoint order-convex sets $U_-$ and $U_+$ of diameter $<\e$ in $\IQ+\sqrt{3}\IQ$ such that $U_-\cap(A'_\pm\cup B'_\pm)=\{z_-\}$, $U_+\cap(A'_\pm\cup B'_\pm)=\{z_+\}$,  and $U_-\cup U_+=\U(z_-)=\U(z_+)$;
\item if a set $U\in\U$ is disjoint with $\check A_\pm\cup\check B_\pm$, then $U$  is order-convex, has diameter $<\e$, and $|U\cap(A'_\pm\cup B'_\pm)|{\le}1$.
\end{enumerate}

\begin{claim}\label{cl3} Any open cover $\U$ of $\IQ+\sqrt{3}\IQ$ can be refined by an $(A',B';\e)$-admissible cover $\A$ of $\IQ+\sqrt{3}\IQ$.
\end{claim}

\begin{proof} By Claim~\ref{cl2}, the sets $A'_\pm$ and $B'_{\pm}$ are closed and discrete in $\IQ+\sqrt{3}\IQ$. Since the set $D=A'_\pm\cup B'_\pm$ is closed and discrete in $\IQ+\sqrt{3}\IQ$, for every $x\in \IQ+\sqrt{3}\IQ$ we can find a neighborhood $O_x\subset\IQ+\sqrt{3}\IQ$ of $x$ such that $O_x$ is contained in some set $U\in\U$, has diameter $<\e$, and $|O_x\cap D|\le 1$. By Claim~\ref{cl1}, the cover $\{O_x:x\in\IQ+\sqrt{3}\IQ\}$ can be refined by a disjoint open cover $\C$ consisting of order-convex subsets of $\IQ+\sqrt{3}\IQ$. Since the cover $\C$ is disjoint and open, each set $C\in\C$ is closed-and-open in $\IQ+\sqrt{3}\IQ$. 
It follows that each set $C\in\C$ has diameter $<\e$ and contains at most one point of the set $D$. Let $\ddot\C=\{C\in \C:C\cap (\check A_\pm\cup\check B_\pm)\ne\emptyset\}$. For every set $C\in\ddot\C$ find a (unique) point $z\in \check A\cup \check B$ such that $C\cap\{z_-,z_+\}\ne\emptyset$. Since $|C\cap\{z_-,z_+\}|=1$, there exists a set $\dot C\in\ddot\C\setminus\{C\}$ such that $\dot C\cap \{z_-,z_+\}\ne\emptyset$. Put $\ddot C=C\cup \dot C$. It is easy to check that the cover $$\A:=(\C\setminus\ddot \C)\cup\{\ddot C:C\in\ddot \C\}$$is $(A',B';\e)$-admissible.
\end{proof}

Using Claim~\ref{cl3}, choose an $(A_0,B_0,1)$-admissible cover $\U_0$ of $\IQ+\sqrt{3}\IQ$. Choose a bijective map $\varphi_0:\U_0\to\U_0$ such that $\varphi_0(\U_0(a_-))=\U_0(f_0(a)_-)$ and $\varphi_0(\U_0(a_+))=\U_0(f_0(a)_+)$ for any $a\in A_0$. Then $f(a)\in \varphi_0^\diamond(\U_0^\diamond(a))$ for each $a\in A_0$.

Let $\preceq$ be any well-order of the countable set $\IB$ such that for every $z\in\IB$ the set $\{b\in\IB:b\preceq z\}$ is finite. For a subset $C\subset \IB$ by $\min C$ we denote the smallest element of the set $C$ in the well-ordered set $(\IB,\preceq)$.

By induction we shall construct:
\begin{itemize}
\item increasing sequences $(A_n)_{n\in\w}$ and $(B_n)_{n\in\w}$ of $\theta$-discrete subsets of $\IB$;\,
\item sequences of points $(a_n)_{n\in\IN}$, $(a_n')_{n\in\IN}$, $(b_n)_{n\in\IN}$, $(b_n')_{n\in\IN}$ of $\IB$,
\item a sequence of bijective functions $(f_n:A_n\to B_n)_{n\in\w}$,
\item a sequence $(\U_n)_{n\in\w}$ of $(A_n,B_n;2^{-n})$-admissible covers $\U_n$ of $\IQ+\sqrt{3}\IQ$,
\item a sequence of bijective functions $(\varphi_n:\U_n\to \U_n)_{n\in\w}$,
\end{itemize}
such that for every $n\in\IN$ the following conditions are satisfied:
\begin{enumerate}
\item $a_n=\min (\IB\setminus A_{n-1})$;
\item $b_n\in \varphi^\diamond_{n-1}(\U^\diamond_{n-1}(a_n))\setminus B_{n-1}$;
\item $b_n\in\IB_0$ if and only if $a_n\in\IB_0$;
\item $b_n'=\min \big(\IB\setminus (B_{n-1}\cup\{b_n\})\big)$;
\item $a_n'\in \IB\setminus (A_{n-1}\cup\{a_n\})$;
\item $a_n'\in\IB_0$ if and only if $b_n'\in\IB_0$;
\item $b_n'\in\varphi^\diamond_{n-1}(\U^\diamond_{n-1}(a_n'))$;
\item $A_n=A_{n-1}\cup \{a_n,a_n'\}$;
\item $B_n=B_{n-1}\cup \{b_n,b_n'\}$;
\item $f_n{\restriction}A_{n-1}=f_{n-1}$, $f_n(a_n)=b_n$, $f_n(a_n')=b_n'$;
\item each set $V\in\U_n$ is contained in some set $U\in\U_{n-1}$;
\item for every $U\in\U_{n-1}$ the family $\{V\in\U_n:V\subset U\}$ is infinite;
\item for every $U\in\U_{n-1}$ and $V\in\U_n$ with $V\subset U$ we have $\varphi_n(V)\subset\varphi_{n-1}(U)$; 
\item for every $x\in A_n$ we get $f_n(x)\in \varphi_n^\diamond(\U_n^\diamond(x))$.
\end{enumerate}

Assume that for some $n\in\IN$ and every $k<n$ we have constructed $\theta$-discrete sets $A_k,B_k$, a bijective function $f_k:A_k\to B_k$, an $(A_k,B_k;2^{-k})$-admissible cover $\U_k$ of $\IB$ and a bijective function $\varphi_k:\U_k\to\U_k$ that satisfy the inductive conditions $(1)$--$(14)$. Let $a_n$ be the point determined by the condition (1). 

If $a_n\in\IB_0$, then $a_n\in \U_{n-1}(a_n)\diamond\U_{n-1}(a_n)$ and $\varphi_{n-1}^\diamond(\U_{n-1}^\diamond(a_n))=\varphi_{n-1}(\U_{n-1}(a_n))\diamond \varphi_{n-1}(\U_{n-1}(a_n))$. It follows that the intersection $\varphi_{n-1}^\diamond(\U_{n-1}^\diamond(a_n))\cap \IB_0$ is a non-empty open set in $\IB_0$. Since $\IB_0$ has no isolated points and the set $\IB_0\cap B_{n-1}$ is discrete in $\IB_0$, there exists an element $b_n\in \IB_0\cap \varphi_{n-1}^\diamond(\U_{n-1}^\diamond(a_n))\setminus B_{n-1}$.
If $a_n\notin\IB_0$, then the set $\varphi_{n-1}^\diamond(\U_{n-1}^\diamond(a_n))\setminus (\IB_0\cup B_{n-1})$ is infinite and hence contains a point $b_n$. 

Let $b_n'$ be the point determined by the condition (5). By analogy with the choice of the point $b_n$, choose a point $a_n'$ satisfying the conditions (6), (7). Define the sets $A_n$ and $B_n$ using the conditions (8), (9) and the bijective map $f_n:A_n\to B_n$ using the condition (10). The set $A_n$ is $\theta$-discrete, being the union of the $\theta$-discrete set $A_{n-1}$ and the finite (and hence $\theta$-discrete) set $\{a_n,a_n'\}$. By analogy we can show that the set $B_n=B_{n-1}\cup\{b_n,b_n'\}$ is $\theta$-discrete.

Applying Claim~\ref{cl3}, choose an $(A_n,B_n;2^{-n})$-admissible cover $\U_n$ satisfying the conditions (11) and (12).

Now we define a bijective map $\varphi_n:\U_n\to\U_n$ satisfying the conditions (13) and (14). Write the sets $A_n$ and $B_n$ as $A_n=\check A\cup \hat A\cup \dot A_n\cup\ddot A_n$ and $B_n=\check B\cup \hat B\cup \dot B_n\cup\ddot B_n$
where $$\dot A_n=A_n\cap\IB_0\setminus(\check A\cup\hat A),\;\;\ddot A_n= A_n\setminus(\IB\cup\check A\cup\hat A),\;\;\ddot B_n=B_n\cap\IB_0\setminus(\check B\cup\hat B),\;\;\ddot B_n= B_n\setminus(\IB_0\cup\check B\cup\hat B).$$
The definition of the sets $\check A,\hat A$, $\check B$ and $\hat B$ and the inductive conditions (3,6,10) ensure that $$f_n(\check A)=\hat B,\;\;f_n(\hat A)=\check B,\;\;f_n(\ddot A_n)=\ddot B_n,\;\;f_n(\dot A_n)=\dot B_n.$$
In the cover $\U_n$ consider the subfamilies:
\begin{itemize}
\item[] $\check \V_n=\{U\in\U_n:\exists a\in \check A\;(a_\pm\subset U)\}$, \ $\check \W_n=\{U\in\U_n:\exists b\in \check B\;(b_\pm\subset U)\}$, \ , \;\; 
\item[] $\hat \V_n=\{U\in\U_n:\exists a\in \hat A\;(a\in U)\}$, \;\; $\hat \W_n=\{U\in\U_n:\exists b\in \hat B\;(b\in U)\}$
\item[] $\dot\V_n=\{U\in\U_n:\exists a\in\dot A_n\;\;(a\in U)\}$,\;\;$\dot\W_n=\{U\in\U_n:\exists b\in\dot B_n\;\;(b\in U)\}$;
\item[] $\ddot\V^-_n=\{U\in\U_n:\exists a\in\ddot A_n\;\;(a_-\in U)\}$, \ $\ddot\W^-_n=\{U\in\U_n:\exists b\in\ddot B_n\;\;(b_-\in U)\}$;
\item[] $\ddot\V^+_n=\{U\in\U_n:\exists a\in\ddot A_n\;\;(a_+\in U)\}$, \ $\ddot\W^+_n=\{U\in\U_n:\exists b\in\ddot B_n\;\;(b_+\in U)\}$.
\end{itemize}
The $(A_n,B_n;2^{-n})$-admissibility of the cover $\U_n$ ensures that the families 
$\check \V_n,\hat\V_n,\dot\V_n,\ddot\V_n^-,\ddot\V_n^+$ are pairwise disjoint. By the same reason, the families $\check \W_n,\hat\W_n,\dot\W_n,\ddot\W_n^-,\ddot\W_n^+$ are pairwise disjoint. Let $$\V_n:=\U_n\setminus(\check \V_n\cup \hat\V_n\cup \dot\V_n\cup\ddot\V_n^-\cup\ddot\V_n^+)\mbox{ \ and \ }
\W_n:=\U_n\setminus(\check \W_n\cup \hat\W_n\cup \dot\W_n\cup\ddot\W_n^-\cup\ddot\W_n^+).
$$ 

The $(A_{n-1},B_{n-1},2^{n-1})$-admissibility of the cover $\U_{n-1}$ implies that every set $U\in\U_{n-1}$ has finite intersection with the set $(A_n\cup B_n)_{\pm}$, which implies that the family $\U_n[U]=\{V\in\U_n:V\subset U\}$ has finite  intersection with the families $\check \V_n\cup \hat\V_n\cup \dot\V_n\cup\ddot\V_n^-\cup\ddot\V_n^+$ and $\check \W_n\cup \hat\W_n\cup \dot\W_n\cup\ddot\W_n^-\cup\ddot\W_n^+$. Combining this fact with the inductive condition (12), we conclude that the families $\U_n[U]\cap\V_n$ and $\U_n[U]\cap\W_n$ are infinite. Then we can define a bijective function $\varphi_n:\U_n\to\U_n$ satisfying the following conditions:
\begin{itemize}
\item [(a)] for any $U\in\U_{n-1}$ we have $\varphi_n(\U_n[U]\cap \V_n)=\U_n[\varphi_{n-1}(U)]\cap\W_n$ ;
\item [(b)] for any $U\in\check \V_n$ there exists a unique $a\in \check A$ with $a_\pm\subset U$ and then $\varphi_n(U)$ is the unique set in $\hat \W_n$ containing the point $f_0(a)\in B\cap\IB_0$;
\item [(c)] for any $U\in\hat \V_n$ there exists a unique $a\in \hat A$ with $a\in U$ and then $\varphi_n(U)$ is the unique set in $\check \W_n$ containing the doubleton $f_0(a)_\pm$;
\item [(d)] for any $U\in\dot \V_n$ there exists a unique $a\in \dot A_n$ with $a\in U$ and then $\varphi_n(U)$ is the unique set in $\dot \W_n$ containing the point  $f_n(a)\in \dot B_n$;
\item [(e)] for any $U\in\ddot \V^-_n$ there exists a unique $a\in \ddot A_n$ with $a_-\in U$ and then $\varphi_n(U)$ is the unique set in $\dot \W_n$ containing the point  $f_n(a)_-$;
\item [(f)] for any $U\in\ddot \V^+_n$ there exists a unique $a\in \ddot A_n$ with $a_+\in U$ and then $\varphi_n(U)$ is the unique set in $\dot \W_n$ containing the point  $f_n(a)_+$.
\end{itemize}

Such choice of the bijection $\varphi_n$ ensures that it satisfies the inductive conditions (13) and (14). This completes the inductive step.
\smallskip

After completing the inductive construction, observe that the conditions (1) and (4) of the inductive construction imply that $\IB=\bigcup_{n\in\w}A_n=\bigcup_{n\in\w}B_n$. Next, consider the bijective map $f:\IB\to\IB$ defined by $f|A_n=f_n$ for all $n\in\w$ and observe that $f{\restriction}A_0=f_0$. 
Using the $(A_n,B_n,2^{-n})$-admissibility of the covers $\U_n$ and  the conditions (1)--(14), it can be shown that the map $f$ is a homeomorphism.

\end{document}